\documentclass[a4paper,12pt]{amsart}
\usepackage{amsmath,amstext,amsbsy,amssymb,amsthm, eufrak}
\usepackage[all]{xy}
\usepackage[mathlines,pagewise]{lineno}
\newtheorem{theorem}{Theorem}[section]
\newtheorem{definition}[theorem]{Definition}

\newtheorem{lemma}[theorem]{Lemma}
\newtheorem{corollary}[theorem]{Corollary}
\theoremstyle{definition}
\newtheorem{example}[theorem]{Example}
\newcommand{\perm}[1]{\mbox{Perm}(#1)}
\newcommand{\Hol}[1]{\mbox{Hol}(#1)}
\newcommand{\Aut}[1]{\mbox{Aut}(#1)}
\newcommand{\Gal}[1]{\mbox{Gal}(#1)}
\begin{document}

\def\keywords#1{\def\@keywords{#1}}

\title[Isomorphism problems for Hopf algebras]{Isomorphism problems for Hopf-Galois structures on separable field extensions}

\author{Alan Koch}
\address{Department of Mathematics, Agnes Scott College, 141 E. College Ave., Decatur, GA 30030 USA}
\email{akoch@agnesscott.edu}
\author{Timothy Kohl}
\address{Department of Mathematics and Statistics, Boston University, 111 Cummington Mall, Boston, MA 02215 USA}
\email{tkohl@math.bu.edu} 
\author{Paul J.~Truman}
\address{School of Computing and Mathematics, Keele University, Staffordshire, ST5 5BG, UK}
\email{P.J.Truman@Keele.ac.uk}
\author{Robert Underwood}
\address{Department of Mathematics and Computer Science, Auburn University at Montgomery, Montgomery, AL, 36124 USA}
\email{runderwo@aum.edu}


\maketitle

\begin{abstract}
Let $ L/K $ be a finite separable extension of fields whose Galois closure $ E/K $ has group $ G $. Greither and Pareigis have used Galois descent to show that a Hopf algebra giving a Hopf-Galois structure on $ L/K $ has the form $ E[N]^{G} $ for some group $ N $ such that $ |N|=[L:K] $. We formulate criteria for two such Hopf algebras to be isomorphic as Hopf algebras, and provide a variety of examples. In the case that the Hopf algebras in question are commutative, we also determine criteria for them to be isomorphic as $ K $-algebras. By applying our results, we complete a detailed analysis of the distinct Hopf algebras and $ K $-algebras that appear in the classification of Hopf-Galois structures on a cyclic extension of degree $ p^{n} $, for $ p $ an odd prime number. 
\end{abstract}
\noindent {\it key words:} Hopf-Galois extension, Greither-Pareigis theory, Galois descent\par
\par

\section{Introduction} \label{section_introduction}

Let $ L/K $ be a finite extension of fields and $ H $ a $ K $-Hopf algebra. We say that $ L $ is an {\em $ H $-Galois extension of $ K $}, or that $ H $ gives a {\em Hopf-Galois structure} on $ L/K $, if $ L $ is an $ H $-module algebra and the obvious $ K $-linear map $ L \otimes_{K} H \rightarrow \mathrm{End}_{K}(L) $ is bijective. For example, if $ L/K $ is a Galois extension with group $ G $ then the group algebra $ K[G] $, with action induced from the usual action of $ G $ on $ L $, gives a Hopf-Galois structure on $ L/K $. We say that two Hopf algebras $ H_{1}, H_{2} $ give {\em isomorphic} Hopf-Galois structures on a finite extension $ L/K $ if there is an isomorphism of $ K $-Hopf algebras $ \varphi : H_{1} \rightarrow H_{2} $ such that $ h x = \varphi(h) x $ for all $ h \in H_{1} $ and $ x \in L $. Note that it is possible for two distinct Hopf-Galois structures on $ L/K $ to have underlying Hopf algebras which are isomorphic as $ K $-Hopf algebras or as $ K $-algebras; equivalently, one might view this as multiple actions of a single $ K $-Hopf algebra or $ K $-algebra on $ L $. In this paper we study this phenomenon. 
\\ \\
If $ L/K $ is purely inseparable, it is known that a single Hopf algebra can act in an infinite number of ways: see e.g. \cite{Koch2014}. We shall therefore suppose that $ L/K $ is separable. In this case Greither and Pareigis \cite{GreitherPareigis1987} have classified the (isomorphism classes of) Hopf-Galois structures admitted by $ L/K $. In order to state this classification we require some notation. Let $ E $ be the Galois closure of $ L/K $, $ G = \Gal{E/K} $, and $ G_{L} = \Gal{E/L} $. Let $ X $ denote the left coset space $ G/G_{L} $, and define a homomorphism $ \lambda : G \rightarrow \perm{X} $ by $ \lambda(\sigma)(\overline{\tau}) = \overline{\sigma \tau} $, where $ \overline{\tau} $ denotes the coset $ \tau G_{L} \in X $. The theorem of Greither and Pareigis asserts that there is a bijection between Hopf-Galois structures on $ L/K $ and  subgroups $ N $ of $ \perm{X} $ which are regular (that is, having the same size as $ X $ and acting transitively on $ X $) and normalized by $ \lambda(G) $ (that is, stable under the action of $ G $ on $ \perm{X} $ defined by $ \sigma \ast \eta = \lambda(\sigma) \eta \lambda(\sigma)^{-1} $). The enumeration of the Hopf-Galois structures admitted by $ L/K $ is therefore equivalent to the enumeration of subgroups of $ \perm{X} $ with these properties. If $ |X| $ is large then this is a difficult problem, but Byott's translation theorem \cite{Byott1996} provides a useful simplification. Loosely, for each abstract group $ N $ of order $ |X| $, Byott's theorem relates the number of $ G $-stable regular subgroups of $ \perm{X} $ that are isomorphic to $ N $ to the number of subgroups of the holomorph of $ N $ that are isomorphic to $ G $. Since $ \Hol{N} \cong N \rtimes \Aut{N} $, this group is much smaller than $ \perm{X} $. We give a more precise statement of Byott's theorem in subsection \ref{subsection_Hopf_algebra_isomorphisms_holomorph} below. 
\\ \\
The theorem of Greither and Pareigis also asserts that the Hopf algebra appearing in the Hopf-Galois structure corresponding to the $ G $-stable regular subgroup $ N $ of $ \perm{X} $ is $ E[N]^{G} $, the fixed points of the group algebra $ E[N] $ under the simultaneous action of $ G $ on $ E $ as Galois automorphisms and on $ N $ by the action $ \ast $. We will refer to the isomorphism class of $ N $ as the {\em type} of Hopf-Galois structure given by $ E[N]^{G} $. The proof that $ E[N]^{G} $ is indeed a $ K $-Hopf algebra, and the proof that it gives a Hopf-Galois structure on $ L/K $, are accomplished via Galois descent. We review this theory briefly since it will feature in many of our arguments. 
\\ \\
The field $ E $ is a finite dimensional $ K $-vector space and so, by Morita theory (see, e.g. \cite[\S 3D]{Curtis_Reiner_I}), the functor $ E \otimes_{K} - $ is an equivalence between the category of $ K $-vector spaces and the category of $ R = \mathrm{End}_{K}(E) $-modules, with inverse $ \mathrm{Hom}_{R}(E,-) $. But since $ E $ is a Galois extension of $ K $ with group $ G $, we may identify $ R $ with the skew group ring $ E \circ G $ \cite{ChaseHarrisonRosenberg1965}, and so an $ R $-module is simply an $ E $-vector space with a compatible $ G $-action. In this case the functor  $ \mathrm{Hom}_{R}(E,-) $ is naturally isomorphic to the fixed point functor $ (-)^{G} $, and so homomorphisms of $ E $-vector spaces descend to homomorphisms of $ K $-vector spaces if and only if they respect the $ G $-action. Applying this to the structure maps defining an $ E $-algebra (resp. an $ E $-Hopf algebra), we obtain criteria for us to descend to a $ K $-algebra (resp. $ K $-Hopf algebra). In the case of the objects appearing in the theorem of Greither and Pareigis, the group algebra $ E[N] $ is certainly an $ E $-Hopf algebra, and we may verify that the Hopf algebra structure maps do respect the action of $ G $ on $ E[N] $ defined above. We therefore obtain that $ E[N]^{G} $ is indeed a $ K $-Hopf algebra. Since the inverse to the fixed point functor is the base change functor, we also obtain that $ E \otimes_{K} E[N]^{G} = E[N] $.
\\ \\
Now suppose that $ N_{1}, N_{2} $ are $ G $-stable regular subgroups of $ \perm{X} $, so that $ H_{1}=E[N_{1}]^{G} $ and $ H_{2}=E[N_{2}]^{G} $ are two Hopf algebras giving Hopf-Galois structures on $ L/K $. In section \ref{section_Hopf_algebra_isomorphisms} we seek criteria for $ H_{1} \cong H_{2} $ as $ K $-Hopf algebras. Since $ E \otimes_{K} H_{i} = E[N_{i}] $ for each $ i $, a necessary condition for $ H_{1} \cong H_{2} $ as $ K $-Hopf algebras is that $ E[N_{1}] \cong E[N_{2}] $ as $ E $-Hopf algebras, which occurs if and only if $ N_{1} \cong N_{2} $ as groups. However, this condition is not sufficient: for this, we require an isomorphism $ N_{1} \xrightarrow{\sim} N_{2} $ that respects the action of $ G $ on each of these groups. This idea appears in, for example, \cite{Childs2013} and \cite{Crespo2015}. We prove a more general result of this type (Theorem \ref{theorem_Hopf_algebra_isomorphism}), and provide a variety of examples. We show that it is possible to detect $ K $-Hopf algebra isomorphisms by studying properties of $ \Hol{N_{1}} $ and $ \Hol{N_{2}} $ (Theorem \ref{theorem_Hopf_algebra_isomorphism_holomorph}), and also determine a criterion for $ F \otimes_{K} H_{1} $ and $ F \otimes_{K} H_{2} $ to be isomorphic as $ F $-Hopf algebras, for $ F $ some extension of $ K $ contained in $ E $ (Theorem \ref{theorem_Hopf_algebra_isomorphism_base_change}). 
\\ \\
In section \ref{section_algebra_isomorphisms_commutative} we assume that $ K $ has characteristic zero and that $ H_{1} $ and $ H_{2} $ are commutative (equivalently, that $ N_{1}, N_{2} $ are abelian groups). We determine a criterion, in terms of the dual groups $ \widehat{N_{1}} $ and $ \widehat{N_{2}} $, for $ H_{1} \cong H_{2} $ as $ K $-algebras  (Theorem \ref{theorem_algebra_isomorphism}). If $ N_{1} \cong N_{2} $, we show that it is also possible to detect such isomorphisms by studying $ \Hol{N_{1}} $ and $ \Hol{N_{2}} $ (Theorem \ref{theorem_algebra_isomorphism_holomorph}). We show that these results have a particularly simple form in the case that $ N_{1} $ and $ N_{2} $ are both cyclic of order $ n $ and $ K $ contains a primitive $ n^{th} $ root of unity (Theorem \ref{corollary_algebra_isomorphism_cyclic}). Finally, in section \ref{section_cyclic_p_power_extensions} we apply the results of the preceding sections to give a detailed analysis of the Hopf-Galois structures admitted by a cyclic extension $ L/K $ of odd prime power degree. We show that the Hopf algebras that appear are pairwise nonisomorphic as Hopf algebras (Theorem \ref{theorem_cyclic_p_power_Hopf_algebra_isomorphisms}) and, under the assumption that $ K $ has characteristic zero and contains a primitive $ p^{n} $-root of unity,  determine the $ K $-algebra isomorphism classes (Theorem \ref{theorem_cyclic_p_power_wedderburn}). 

\section{Hopf Algebra Isomorphisms} \label{section_Hopf_algebra_isomorphisms}

In this section we address the question of when two Hopf algebras giving Hopf-Galois structures on a finite separable extension of fields are isomorphic as Hopf algebras. We achieve this by exploiting that fact that, by the theorem of Greither and Pareigis, such a Hopf algebra is the fixed points of a group algebra under the action of a certain group of automorphisms. Rather than focusing specifically on this situation, we make our definitions, and formulate our first theorem, in rather more general terms:

\begin{definition}
Let $ G $ be a group and let $ (N_{1},\ast_{1}), (N_{2},\ast_{2}) $ be $ G $-sets (where, for $ =1,2 $, $ \ast_{i} $ denotes the action of $ G $ on $ N_{i} $). We say that $ (N_{1},\ast_{1}),  (N_{2},\ast_{2}) $ (or just $ N_{1},N_{2} $) are {\em isomorphic as $ G $-sets} if there is a $ G $-equivariant bijection $ f : N_{1} \rightarrow N_{2} $. We say that a $ G $-set $ N $ is a {\em $ G $-group} if it is a group on which $ G $ acts via automorphisms, and that two $ G $-groups $ N_{1}, N_{2} $ are {\em isomorphic as $ G $-groups} if there is a $ G $-equivariant isomorphism $ f : N_{1} \xrightarrow{\sim} N_{2} $. 
\end{definition}

We note that if $ \varphi: N_{1} \xrightarrow{\sim} N_{2} $ is an isomorphism of $ G $-sets and $ \eta_{1}, \ldots ,\eta_{r} $ are representatives for the orbits of $ G $ in $ N_{1} $ then $ \varphi(\eta_{1}), \ldots ,\varphi(\eta_{r}) $ are representatives for the orbits of $ G $ in $ N_{2} $, and $ \mbox{Stab}(\eta_{i}) = \mbox{Stab}(\varphi(\eta_{i})) $ for each $ i $. We use these definitions to formulate the following general theorem:

\begin{theorem} \label{theorem_Hopf_algebra_isomorphism}
Let $ E/K $ be a Galois extension of fields, let $ G $ be a subgroup of $ \Gal{E/K} $, and let $ F=E^{G} $. Let $ (N_{1}, \ast_{1}) $ and $ (N_{2}, \ast_{2}) $ be $ G $-groups and, for $ i=1,2 $, let $ G $ act on $ E[N_{i}] $ by acting on $ E $ as Galois automorphisms and on $ N_{i} $ via $ \ast_{i} $. Then $ E[N_{1}]^{G} \cong  E[N_{2}]^{G} $ as $ F $-Hopf algebras if and only if $ N_{1} \cong N_{2} $ as $ G $-groups.  
\end{theorem}
\begin{proof}
Note that $ E/F $ is a Galois extension of fields with group $ G $. If there is a $ G $-equivariant isomorphism of groups $ \varphi : N_{1} \xrightarrow{\sim} N_{2} $, then this extends to a $ G $-equivariant isomorphism of $ E $-Hopf algebras $ \varphi : E[N_{1}] \xrightarrow{\sim} E[N_{2}] $, which (since it is $ G $-equivariant) descends to an isomorphism of $ F $-Hopf algebras $ \varphi : E[N_{1}]^{G} \xrightarrow{\sim}  E[N_{2}]^{G} $. 
\\ \\ 
Conversely, suppose that $ \psi : E[N_{1}]^{G} \xrightarrow{\sim} E[N_{2}]^{G} $ is an isomorphism of $ F $-Hopf algebras. Since we have $ E \otimes_{F} E[N_{i}]^{G} $ for each $ i $, it extends to an isomorphism of $ E $-Hopf algebras $ \psi : E[N_{1}] \xrightarrow{\sim} E[N_{2}] $. Restricting to the grouplike elements on both sides we obtain an isomorphism of groups $ \psi : N_{1} \xrightarrow{\sim} N_{2} $. Thus, we need to show that $ \psi $ is $ G $-equivariant. Let $ \eta \in N_{1} $, and let $ \{ a_{i} \} $ be an $ F $-basis of $ E[N_{1}]^{G} $. Then $ \{ a_{i} \} $ is also an $ E $-basis of $ E[N_{1}] $, so there exist unique $ x_{i} \in E $ such that
\[ \eta = \sum_{i} x_{i} a_{i}. \]
Now for $ g \in G $ we have
\begin{align*}
\psi(g \ast_{1} \eta) & = \psi \left(g \ast_{1} \sum_{i} x_{i} a_{i} \right) \\
& = \psi \left( \sum_{i} g(x_{i})  (g \ast_{1} a_{i}) \right) \\
& = \psi \left( \sum_{i} g(x_{i}) a_{i} \right) \\
& = \sum_{i} g(x_{i}) \psi(a_{i}) \\
& = \sum_{i}g( x_{i})(g \ast_{2} \psi(a_{i})) \\
&= g \ast_{2} \left( \sum_{i} x_{i} \psi(a_{i}) \right) \\
&= g \ast_{2} \psi \left( \sum_{i} x_{i} a_{i} \right) \\
& = g \ast_{2} \psi( \eta).
\end{align*}
Therefore $ \psi(g \ast_{1} \eta) = g \ast_{2} \psi( \eta) $, and so $ \psi : N_{1} \xrightarrow{\sim} N_{2} $ is $ G $-equivariant. 
\end{proof}

The following corollary is the principal application of Theorem \ref{theorem_Hopf_algebra_isomorphism}. We retain the notation established in section \ref{section_introduction} to describe the theorem of Greither and Pareigis: $ L/K $ is a finite separable extension of fields with Galois closure $ E$, $ G = \Gal{E/K} $, $ G_{L} = \Gal{E/L} $, $ X = G/G_{L} $. A Hopf algebra giving a Hopf-Galois structure on $ L/K $ then has the form $ E[N]^{G} $ for some regular subgroup $ N $ of $ \perm{X} $ stable under the action of $ G $ on $ \perm{X} $ by $ \sigma \ast \eta = \lambda(\sigma) \eta \lambda(\sigma)^{-1} $. 

\begin{corollary} \label{corollary_Hopf_algeba_isomorphism}
Let $ E[N_{1}]^{G} $ and $ E[N_{2}]^{G} $ give Hopf-Galois structures on $ L/K $. Then $ E[N_{1}]^{G} \cong E[N_{2}]^{G} $ as $ K $-Hopf algebras if and only if $ N_{1} \cong N_{2} $ as $ G $-groups. 
\end{corollary}

We illustrate the applicability of the theory with a variety of examples. 

\begin{example}{\bf (The classical and canonical nonclassical Hopf-Galois structures)} \label{example_Hopf_isomorphism_classical_nonclassical}
If $ L/K $ is a Galois extension then in the notation of the theorem of Greither and Pareigis we have $ E=L $, $ G_{L} = \{ 1 \} $, and $ X = G $, and the homomorphism $ \lambda : G \rightarrow \perm{G} $ is in fact the left regular embedding of $ G $. In this case examples of $ G $-stable regular subgroups of $ \perm{G}  $ are $ \lambda(G) $ itself and $ \rho(G) $, the image of $ G $ under the right regular embedding (these subgroups coincide if and only if $ G $ is abelian). The elements of $ \rho(G) $ commute with those of $ \lambda(G) $, so the action of $ G $ on $ \rho(G) $ is trivial, and therefore the Hopf algebra appearing in the Hopf-Galois structure corresponding to $ \rho(G) $ is $ L[\rho(G)]^{G} = L^{G}[\rho(G)] = K[\rho(G)] $, which is isomorphic to the Hopf algebra $ K[G] $. We call this the {\em classical Hopf-Galois structure} on $ L/K $. If $ G $ is nonabelian then the subgroup $ \lambda(G) $ corresponds to a different Hopf-Galois structure on $ L/K $, which we call the {\em canonical nonclassical Hopf-Galois structure}. The Hopf algebra appearing in this structure is $ H_{\lambda} := L[\lambda(G)]^{G} $. Since $ G $ is nonabelian the action of $ G $ on $ \lambda(G) $ is not trivial (the orbits are the conjugacy classes in $ \lambda(G) $), and so $ \rho(G) \not \cong \lambda(G) $ as $ G $-groups in this case. Therefore by Corollary \ref{corollary_Hopf_algeba_isomorphism} $ K[G] \not \cong H_{\lambda} $ as $ K $-Hopf algebras. 
\end{example}

\begin{example}{\bf (Elementary abelian extensions of degree $ p^{2} $)} \label{example_Hopf_isomorphism_elementary_abelian_p2}
Let $ p > 2 $ be prime, and let $ L/K $ be a Galois extension of fields of degree $ p^{2} $ with elementary abelian Galois group $ G $. In \cite[Corollary to Theorem 1, part (iii)]{Byott1996} it is shown that $ L/K $ admits $ p^{2} $ Hopf-Galois structures.  By applying Corollary \ref{corollary_Hopf_algeba_isomorphism} we can determine which of these Hopf-Galois structures involve isomorphic Hopf algebras. In \cite[Theorem 2.5]{Byott2002} the regular subgroups of $ \perm{G} $ that yield the Hopf-Galois structures are determined as follows: let $ T $ be a subgroup of $ G $ of order $ p $, and fix elements $ s, t \in G $ such that
\[ T = \langle t \rangle, \hspace{4mm} s^{p}=1_{G}, \hspace{4mm} G = \langle s, t \rangle. \]
Let $ d \in \{0, 1, \ldots ,p-1\} $, and define $ \alpha, \beta \in \perm{G} $ in terms of their actions on a typical element $ s^{k}t^{l} \in G $:
\begin{equation} \label{equation_byott_p2_permutations}
\begin{array}{rcl}
\alpha[s^{k} t^{l}] & = & s^{k} t^{l-1} \\
\beta[s^{k} t^{l}] & = & s^{k-1} t^{l+(k-1)d}.
\end{array} 
\end{equation}
It is easily verified that $\alpha^{p}=\beta^{p}=1 $ and $ \alpha\beta =\beta\alpha $, whence $ N_{T,d}= \langle \alpha, \beta \rangle \cong G $. Moreover, one can show that $ N_{T,d} $ is a regular subgroup of $ \perm{G} $, and that 
\begin{equation} \label{equation_byott_p2_lambda_action}
s \ast \alpha = \alpha, \hspace{4mm} t \ast \alpha = \alpha, \hspace{4mm} s \ast \beta = \alpha^{d} \beta, \hspace{4mm} t \ast \beta=\beta. 
\end{equation}
Thus $ N_{T,d} $ is $ G $-stable, and therefore corresponds to a Hopf-Galois structure on $ L/K $ with Hopf algebra $ H_{T,d} = L[N_{T,d}]^{G} $. If $ d=0 $ then $ N_{T,d}= \rho(G)  $ regardless of the choice of $ T $, and so we obtain the classical Hopf-Galois structure. Taking $ 1 \leq d \leq p-1 $ and letting $ T $ vary through the subgroups of $ G $ of order $ p $, we obtain $ p^{2}-1 $ distinct groups $ N_{T,d} \neq \rho(G) $, giving in total $ p^{2} $  Hopf-Galois structures on $ L/K $. These are all the Hopf-Galois structures on $ L/K $.
\\ \\ 
We claim that two Hopf algebras $ H_{1} = H_{T_{1},d_{1}} $ and $ H_{2} = H_{T_{2},d_{2}} $ are isomorphic as Hopf algebras if and only if $ d_{1}=d_{2}=0 $ or $ d_{1}d_{2} \neq 0 $ and $ T_{1}=T_{2} $. Let 
\[ \begin{array}{ccccl}
N_{1} & = & N_{T_{1},d_{1}} & = & \langle \alpha_{1}, \beta_{1} \rangle \\
N_{2} & = & N_{T_{2},d_{2}} & = & \langle \alpha_{2}, \beta_{2} \rangle,
 \end{array} \]
where $ \alpha_{1},\beta_{1} $ and $ \alpha_{2},\beta_{2} $ are defined as in Equations \eqref{equation_byott_p2_permutations}, using $ d_{1}, d_{2} $ as appropriate. 
We have seen that if $ d_{1}=d_{2}=0 $ then $ H_{1} = H_{2} $. If $ d_{1}d_{2} \neq 0 $ and $ T_{1}=T_{2} $ then there exists $ c \in \{1, \ldots ,p-1\} $ such that $ cd_{2}\equiv d_{1} \pmod{p} $. Using this, define a homomorphism $ \varphi: N_{1} \rightarrow N_{2} $ by
\[ \varphi(\alpha_{1}) = \alpha_{2}, \hspace{4mm} \varphi(\beta_{1}) = \beta_{2}^{c}. \]
It is clear that $ \varphi $ is an isomorphism, and we claim that it is $ G $-equivariant. We have: 
\[ \begin{array}{l}
\varphi(s \ast \alpha_{1} ) = \varphi( \alpha_{1} ) = \alpha_{2} = s \ast \alpha_{2} =  s \ast \varphi(\alpha_{1}), \\
\varphi(t \ast \alpha_{1} ) = \varphi( \alpha_{1} ) = \alpha_{2} = t \ast \alpha_{2} = t \ast \varphi(\alpha_{1}), \\
\varphi(s \ast \beta_{1}) = \varphi(\alpha_{1}^{d_{1}}\beta_{1}) = \alpha_{2}^{d_{1}} \beta_{2}^{c} = \alpha_{2}^{d_{2}c} \beta_{2}^{c}  = s \ast \beta_{2}^{c} = s \ast \varphi(\beta_{1}), \\
\varphi(t \ast \beta_{1} ) = \varphi( \beta_{1} ) = \beta_{2}^{c} = t \ast \beta_{2}^{c} = t \ast \varphi(\beta_{1}). 
\end{array} \]
Thus $ \varphi $ is a $ G $-equivariant isomorphism of $ N_{1} $ onto $ N_{2} $, and so $ H_{1} \cong H_{2} $ as Hopf algebras by Corollary \ref{corollary_Hopf_algeba_isomorphism}. 
\\ \\
For the converse, note that if $ T $ is a subgroup of $ G $ of order $ p $ and $ d \neq 0 $ then by \eqref{equation_byott_p2_lambda_action} the kernel of the action of $ G $ on $ N_{T,d} $ is precisely $ T $. Therefore, if $ d_{1}d_{2} \neq 0 $ and $ N_{1} \cong N_{2} $ as $ G $-groups then we must have $ T_{1} = T_{2} $. 
\\ \\
Thus, the $ p^{2} $ Hopf-Galois structures admitted by $ L/K $ involve exactly $ p+2 $ nonisomorphic Hopf algebras: the group algebra $ K[G] $ with its usual action on $ L $ and, for each of the $ p+1 $ subgroups $ T $ of $ G $ of order $ p $, a Hopf algebra $ H_{T}=H_{T,1} $, acting in $ p-1 $ different ways. 
\end{example}

\begin{example}{\bf (Fixed point free endomorphisms)} \label{example_Hopf_isomorphism_fpf}
If $ L/K $ be a finite Galois extension of fields with nonabelian Galois group $ G $ then, as described above, $ L/K $ admits a canonical nonclassical Hopf-Galois structure with Hopf algebra $ H_{\lambda} = L[\lambda(G)]^{G} $. In \cite{Childs2013}, Childs shows how certain endomorphisms of $ G $ can yield further Hopf-Galois structures on $ L/K $, whose Hopf algebras are isomorphic to $ H_{\lambda} $. Specifically, let $ \psi $ be an endomorphism of $ G $ which is abelian (meaning $ \psi(gh) = \psi(hg) $ for all $ g, h\in G $) and fixed point free (meaning $ \psi(g)=g $ if and only if $ g=1_{G} $). From $ \psi $ we may construct a homomorphism $ \alpha_{\psi}: G \rightarrow \perm{G} $
defined by
\[ \alpha_{\psi}(g) = \lambda(g) \rho(\psi(g)). \]
One can show that $ \alpha_{\psi}(G) $ is a $ G $-stable regular subgroup of $ \perm{G} $, which therefore corresponds to a Hopf-Galois structure on $ L/K $.  It is shown in \cite[Theorem 5]{Childs2013} that the Hopf algebras appearing in the Hopf-Galois structures produced by this construction are all isomorphic to $ H_{\lambda} $ as Hopf algebras. We can reinterpret these ideas via Corollary \ref{corollary_Hopf_algeba_isomorphism}.
\\ \\
We know that $ \lambda(G) $ and $ \alpha_{\psi}(G) $ are both regular subgroups of $ \perm{G} $ normalized by $ \lambda(G) $. Define a map $ \varphi : \lambda(G) \rightarrow \alpha_{\psi}(G) $ by
\[ \varphi(\lambda(g)) = \alpha_{\psi}(g) = \lambda(g) \rho(\psi(g)). \]
It is clear that $ \varphi $ is an isomorphism of groups. In fact, it is also $ G $-equivariant: if $ h\in G $ then we have
\begin{align*}
h \ast \varphi(\lambda(g)) & = \lambda(h) \lambda(g) \rho(\psi(g)) \lambda(h)^{-1} \\
& = \lambda(h) \lambda(g) \lambda(h)^{-1} \rho(\psi(g)) \mbox{ (since $\lambda(G) $ and $ \rho(G) $ commute inside $\perm{G} $)} \\
& = \lambda(hg h^{-1}) \rho(\psi(g)) \\
& = \lambda(hg h^{-1}) \rho(\psi(hg h^{-1})) \mbox{ (since $\psi $ abelian implies that $\psi(g) = \psi(hg h^{-1})$).} \\
& = \varphi(\lambda(hg h^{-1})) \\
&= \varphi(h \ast \lambda(g)). 
\end{align*}
Thus $ \varphi : \lambda(G) \rightarrow \alpha_{\psi}(G) $ is a $ G $-equivariant isomorphism of groups, and so by Corollary \ref{corollary_Hopf_algeba_isomorphism} $ L[\alpha_{\psi}(G)]^{G} \cong L[\lambda(G)]^{G} $ as Hopf algebras. 
\end{example}

\begin{example}{\bf (Conjugating regular subgroups by elements of $ \rho(G) $)} \label{example_Hopf_isomorphism_conjugation}
Let $ L/K $ be a Galois extension of fields with nonabelian Galois group $ G $, and let $ L[N]^{G} $ give a Hopf-Galois structure on $ L/K $. Since $ G $ is nonabelian, we have $ \lambda(G) \neq \rho(G) $ and so, although $ N $ is normalized by $ \lambda(G) $, it may not be normalized by $ \rho(G) $. Suppose that $ g \in G $ is such that $ N_{g} = \rho(g)N\rho(g)^{-1} \neq N $. Then $ N_{g} $ is a regular subgroup of $ \perm{G} $ since $ N $ is, and it is $ G $-stable: if $ h \in G $ then
\begin{align*}
h \ast N_{g} & = h \ast\rho(g)N\rho(g)^{-1} \\
& = \lambda(h) \rho(g)N\rho(g)^{-1} \lambda(h)^{-1} \\
& = \rho(g)\lambda(h) N \lambda(h)^{-1}\rho(g)^{-1} \\
& = \rho(g)N\rho(g)^{-1} \\
&= N_{g}. 
\end{align*}
Therefore $ N_{g} $ corresponds to a Hopf-Galois structure on $ L/K $, with Hopf algebra $ L[N_{g}]^{G} $. The map $ \varphi : N \rightarrow N_{g} $
defined by
\[ \varphi(\eta) = \rho(g) \eta \rho(g)^{-1} \]
is an isomorphism of groups, and is $ G $-equivariant since if $ h \in G $ then
\begin{align*}
h \ast \varphi(\eta) & =h \ast \rho(g) \eta \rho(g)^{-1} \\
&= \lambda(h) \rho(g) \eta \rho(g)^{-1} \lambda(h)^{-1} \\
&= \rho(g)\lambda(h)  \eta  \lambda(h)^{-1}\rho(g)^{-1} \\
&= \rho(g)(h \ast \eta)\rho(g)^{-1} \\
& = \varphi(h \ast \eta).
\end{align*}
Therefore by Corollary \ref{corollary_Hopf_algeba_isomorphism} $ L[N_{g}]^{G} \cong L[N]^{G} $ as Hopf algebras.
\end{example}

\begin{example}{\bf (A specific example of conjugating by elements of $ \rho(G) $) \\}
Let $ p,q $ be primes with $ p \equiv 1 \pmod{q} $, and let $ L/K $ be a Galois extension of fields with group isomorphic to the metacyclic group of order $ pq $:
\[ G = \langle \sigma, \tau \mid \sigma^{p}=\tau^{q}=1, \; \tau \sigma = \sigma^{g} \tau \rangle, \]
where $ g $ is a fixed positive integer whose order modulo $ p $ is $ q $. By \cite[Theorem 6.2]{Byott2004}, $ L/K $ admits precisely $ 2+p(2q-3) $ Hopf-Galois structures: the classical structure, the canonical nonclassical structure, $ 2p(q-2) $ further structures of metacyclic type, and $ p $ structures of cyclic type. We can use Example \ref{example_Hopf_isomorphism_conjugation} to show that the Hopf algebras appearing in the Hopf-Galois structures of cyclic type are all isomorphic as Hopf algebras. 
The regular subgroups $ N_{c} $ ($c=0, \ldots ,p-1$) of $ \perm{G} $ corresponding to these Hopf-Galois structures are described explicitly in \cite[Lemma 4.1]{Byott2004}, each in terms of two generators. Using these descriptions, we can verify that $ N_{c} = \langle \eta_{c} \rangle $, where
\[ \eta_{c}[\sigma^{u} \tau^{v}] = \sigma^{u+1-cg^{v}} \tau^{v+1}. \]
That is:
\[ \eta_{c} = \lambda(\sigma) \rho(\sigma^{-c}\tau)^{-1}. \]
Using the fact that $ \lambda(G) $ and $ \rho(G) $  commute inside $ \perm{G} $ we have in particular, for $ c=0 $,
\begin{eqnarray*}
\rho(\sigma^{i}) \eta_{0} \rho(\sigma^{-i}) & = & \lambda(\sigma) \rho(\sigma^{i} \tau \sigma^{-i} )^{-1} \\
& = & \lambda(\sigma) \rho(\sigma^{i(1-g)}\tau)^{-1} \\
& = & \eta_{i(g-1)},
\end{eqnarray*}
where the subscript should be interpreted modulo $ p $. Since $ g $ has order $ q $ modulo $ p $ we certainly have $ g-1 \neq 0 \pmod{p} $ and so, given $ c= 0, \ldots ,p-1 $ there exists $ i $ such that 
\[  \rho(\sigma^{i}) \eta_{0} \rho(\sigma^{-i}) = \eta_{c}. \] 
Therefore the groups $ N_{c} $ are all isomorphic as $ G $-groups, and so the Hopf algebras $ H_{c} $ are all isomorphic as $ K $-Hopf algebras.
\end{example}

In the case that $ q=2 $ (so that $ G \cong D_{p} $), this result is established in \cite[Proposition 4.3]{TARP_PROMS2017}, by methods different to those employed above. 

\subsection{Hopf algebra isomorphisms after base change} \label{subsection_Hopf-algebra_isomorphism_base_change}

We return to the situation addressed by the theorem of Greither and Pareigis, as described in section \ref{section_introduction}: $ L/K $ is a finite separable extension of fields whose Galois closure $ E $ has group $ G $, so that a Hopf algebra giving a Hopf-Galois structure on $ L/K $ has the form $ E[N]^{G} $ for some $ G $-stable regular subgroup of $ \perm{X} $. If $ E[N_{1}]^{G} $ and $ E[N_{2}]^{G} $ are two such Hopf algebras then, by Galois descent, we have $ E \otimes_{K} E[N_{i}]^{G} = E[N_{i}] $ for each $ i $, so certainly $ E \otimes_{K} E[N_{1}]^{G} $ and $ E \otimes_{K} E[N_{2}]^{G} $ are isomorphic as $ E $-Hopf algebras. However, there may exist intermediate fields $ K^{\prime} $ such that $ K^{\prime} \otimes_{K} E[N_{1}]^{G} $  and $ K^{\prime} \otimes_{K} E[N_{2}]^{G} $ are isomorphic as $ K^{\prime} $-Hopf algebras. In this section we identify the fields with this property that are also Galois over $ K $. 

\begin{theorem} \label{theorem_Hopf_algebra_isomorphism_base_change}
Let $ E[N_{1}]^{G} $ and $ E[N_{2}]^{G} $ give Hopf-Galois structures on $ L/K $, let $ G^{\prime} $ be a normal subgroup of $ G $, and let $ K^{\prime} = E^{G^{\prime}} $. Then
\[ K^{\prime} \otimes E[N_{1}]^{G} \cong K^{\prime} \otimes_{K} E[N_{2}]^{G} \mbox{ as $ K^{\prime} $-Hopf algebras} \]
if and only if $ (N_{1},\ast) \cong (N_{2},\ast) $ as $ G^{\prime} $-groups. 
\end{theorem}
\begin{proof}
For $ i=1,2 $ we have by Galois descent that $ E[N_{i}]^{G^{\prime}} $ is a $ K^{\prime} $-Hopf algebra, and one may show that action of $ G $ on $ E[N_{i}] $ induces an action of the group $ G/G^{\prime} $ on $ E[N_{i}]^{G^{\prime}} $ and that
\[ E[N_{i}]^{G} = \left( E[N_{i}]^{G^{\prime}} \right)^{G/G^{\prime}}. \]
Now since $ K^{\prime}/K $ is Galois with group $ G/G^{\prime} $, we have
\[  K^{\prime} \otimes_{K} \left( E[N_{i}]^{G^{\prime}} \right)^{G/G^{\prime}}  = E[N_{i}]^{G^{\prime}}, \]
again by Galois descent. Therefore 
\[ K^{\prime} \otimes_{K} E[N_{i}]^{G}  = E[N_{i}]^{G^{\prime}} \]
for $ i=1,2 $, and so $ K^{\prime} \otimes_{K} E[N_{1}]^{G}   \cong K^{\prime} \otimes_{K} E[N_{2}]^{G} $ as $ K^{\prime} $-Hopf algebras if and only if $ E[N_{1}]^{G^{\prime}} \cong E[N_{2}]^{G^{\prime}} $ as $ K^{\prime} $-Hopf algebras. By Theorem \ref{theorem_Hopf_algebra_isomorphism} this occurs if and only if $ (N_{1},\ast) \cong (N_{2},\ast) $ as $ G^{\prime} $-groups. 
\end{proof}

\begin{example}{\bf (The smallest extension of scalars giving a group algebra)} \label{example_extension_of_scalars_group_algebra}
Let $ E[N]^{G} $ be a Hopf algebra giving a Hopf-Galois structure on $ L/K $, and let $ G^{\prime} $ denote the kernel of the action of $ G $ on $ N $. Then $ E[N]^{G^{\prime}} = E^{G^{\prime}}[N] = K^{\prime}[N] $, a group algebra with coefficients drawn from the field $ E^{G^{\prime}}=K^{\prime} $. By Theorem \ref{theorem_Hopf_algebra_isomorphism_base_change} we have $ K^{\prime} \otimes_{K} E[N]^{G} =K^{\prime}[N] $. In fact $ K^{\prime} $ is minimal amongst the subfields $ F $ of $ E $ such that $ F \otimes_{K} E[N]^{G} $ is isomorphic as a Hopf algebra to a group algebra (see \cite[Corollary 3.2]{GreitherPareigis1987}).
\end{example}

\begin{example}{\bf (Elementary abelian extensions of degree $ p^{2} $ revisited)}
Let $ p $ be an odd prime, and let $ L/K $ be an elementary abelian extension of degree $ p^{2} $ with group $ G $. In Example \ref{example_Hopf_isomorphism_elementary_abelian_p2} we determined which of the Hopf algebras appearing in the classification of Hopf-Galois structures on $ L/K $ are isomorphic as $ K $-Hopf algebras. Here we can show that, given any two Hopf algebras $ H_{1},H_{2} $ giving nonclassical Hopf-Galois structures on the extension, there exists a subfield $ K^{\prime} $ of $ L/K $ of degree $ p $ over $ K $ such that $ K^{\prime} \otimes_{K} H_{1} \cong K^{\prime} \otimes_{K} H_{2} $ as $ K^{\prime} $-Hopf algebras. 
Recall from Example \ref{example_Hopf_isomorphism_elementary_abelian_p2} that for $ i=1,2 $ the Hopf algebra $ H_{i} $ corresponds to a choice of subgroup $ T_{i} $ of degree $ p $ and an integer $ d_{i} \in \{1, \ldots ,p-1\} $ (the possibility $ d_{i}=0 $ is excluded since the Hopf algebras give nonclassical structures). Specifically, we have $ H_{i} = L[N_{i}]^{G} $, where $ N_{i} $ is generated by two permutations $ \alpha_{i}, \beta_{i} $ as described in Equations \eqref{equation_byott_p2_permutations} and the action of $ G $ on $ N_{i} $ is as described in Equation \eqref{equation_byott_p2_lambda_action}. From this last equation, we see that the kernel of the action of $ G $ on each $ N_{i} $ is precisely $ T_{i} $, which we now write as $ \langle t_{i} \rangle $. Let $ g^{\prime} = t_{1}t_{2} $, $ G^{\prime} = \langle g^{\prime} \rangle $, and $ K^{\prime} = L^{G^{\prime}} $; we claim that $ K^{\prime} \otimes_{K} H_{1} \cong K^{\prime} \otimes_{K} H_{2} $ as $ K^{\prime} $-Hopf algebras. By Theorem \ref{theorem_Hopf_algebra_isomorphism_base_change}, we must show that there is a $ G^{\prime} $-equivariant isomorphism $ \varphi : N_{1} \xrightarrow{\sim} N_{2} $. 
For each $ i $ we have $ g^{\prime} \ast \alpha_{i} = \alpha_{i} $ for each $ i $, and so we focus our attention on the action of $ g^{\prime} $ on the  $ \beta_{i} $. 
Since each $ T_{i} $ is precisely the kernel of the action of $ G $ on $ N_{i} $ and $ G $ is abelian, by equation \eqref{equation_byott_p2_lambda_action} we have 
\[ g^{\prime} \ast \beta_{1} = t_{1}t_{2} \ast \beta_{1} = t_{2} \ast \beta_{1} = \alpha_{1}^{u_{1}} \beta_{1} \mbox{ for some $ u_{1} \in \{ 1, \ldots ,p-1 \} $} \]
and
\[ g^{\prime} \ast \beta_{2}= t_{1}t_{2} \ast \beta_{2} = t_{1} \ast \beta_{2} = \alpha_{2}^{u_{2}} \beta_{2} \mbox{ for some $ u_{2} \in \{ 1, \ldots ,p-1 \} $} .\] 
 There exists an integer $ c $ such that $ cu_{2} \equiv u_{1} \pmod{p} $; use this to define a homomorphism $ \varphi : N_{1} \rightarrow N_{2} $ by
\[ \varphi(\alpha_{1}) = \alpha_{2}, \hspace{4mm} \varphi(\beta_{1})=\beta_{2}^{c}. \]
It is clear that $ \varphi $ is a isomorphism of $ N_{1} $ onto $ N_{2} $, and to verify that it is $ G^{\prime} $-equivariant it suffices to consider the action of $ g^{\prime} $ on  $ \beta_{1} $. We have:
\begin{eqnarray*}
\varphi( g^{\prime} \ast \beta_{1}) & = & \varphi( \alpha_{1}^{u_{1}} \beta_{1} ) \\
& = & \alpha_{2}^{u_{1}} \beta_{2}^{c} \\
& = & \alpha_{2}^{cu_{2}} \beta_{2}^{c} \\
& = & g^{\prime} \ast \beta_{2}^{c}.
\end{eqnarray*}
Therefore $ \varphi $ is a $ G^{\prime} $-equivariant isomorphism of $ N_{1} $ onto $ N_{2} $, and so Theorem \ref{theorem_Hopf_algebra_isomorphism_base_change} we have
\[ K^{\prime} \otimes_{K} H_{1} \cong K^{\prime} \otimes_{K} H_{2} \mbox{ as $ K^{\prime} $-Hopf algebras}. \]
\end{example}

\subsection{Hopf algebra isomorphisms via the holomorph} \label{subsection_Hopf_algebra_isomorphisms_holomorph}

We retain the notation of subsection \ref{subsection_Hopf-algebra_isomorphism_base_change}. We have seen that a necessary condition for $ E[N_{1}]^{G} \cong E[N_{2}]^{G} $ as Hopf algebras is that the underlying groups are isomorphic; we now use this requirement to change our point of view. Fix an abstract group $ N $ of the same order as $ G $,  and study embeddings $ \alpha: N \hookrightarrow \perm{G} $ such that $ \alpha(N) $ is $ G $-stable and regular. Byott's translation theorem \cite[Proposition 1]{Byott1996} relates such embeddings to certain embeddings of $ G $ into the holomorph of $ N $, denoted $ \Hol{N} $, which is normalizer of $ \lambda(N) $ inside $ \perm{N} $. More precisely, there is a bijection between the sets
\[ \{ \alpha : N \hookrightarrow \perm{X} \mid \alpha \mbox{ an injective homomorphism whose image in $ G $-stable and regular} \} \]
and
\[ \{ \beta : G \hookrightarrow \Hol{N} \mid \beta \mbox{ an injective homomorphism such that $ \beta(G_{L}) $ is the stabilizer of $ e_{N} $} \}. \] 
Since $ \Hol{N} = \rho(N) \rtimes \Aut{N} $ \cite[(7.2) Proposition]{Ch00}, it is much smaller than $ \perm{X} $, and so Byott's translation theorem is often useful in the enumeration of the Hopf-Galois structures admitted by a given extension. Of course, different embeddings $ \alpha $ can have the same image (and so correspond to the same Hopf-Galois structure), but this can be detected by studying the corresponding $ \beta $: we have $ \alpha_{1}(N)=\alpha_{2}(N) $ if and only if $ \beta_{1}(G) $ and $ \beta_{2}(G) $ are conjugate by an element of $ \Aut{N} $. In fact, we can also detect when the Hopf-Galois structures corresponding to different embeddings $ \alpha_{1}, \alpha_{2} $ involve isomorphic Hopf algebras by studying properties of the corresponding $ \beta_{1}, \beta_{2} $:

\begin{theorem} \label{theorem_Hopf_algebra_isomorphism_holomorph}
Let $ \alpha_{1}, \alpha_{2} $ be embeddings of $ N $ into $ \perm{X} $ whose images are regular and normalized by $ \lambda(G) $, and let $ \beta_{1}, \beta_{2} $ be the corresponding embeddings of $ G $ into $ \Hol{N} $. Viewing $ \Hol{N} $ as $ \rho(N) \rtimes \Aut{N} $, let $ \overline{\beta_{1}}, \overline{\beta_{2}} $ denote the compositions of $ \beta_{1}, \beta_{2} $ with the projection onto the automorphism component. Then 
\[ E[\alpha_{1}(N)]^{G} \cong E[\alpha_{2}(N)]^{G} \mbox{ as Hopf algebras} \]
if and only if there exists $ \mu \in \Aut{N} $ such that 
\[ \overline{\beta_{2}}(g) = \mu \overline{\beta_{1}}(g) \mu^{-1} \mbox{ for all } g \in G. \]
\end{theorem}
\begin{proof}
By Theorem \ref{theorem_Hopf_algebra_isomorphism}, we have $ E[\alpha_{1}(N)]^{G} \cong E[\alpha_{2}(N)]^{G} $ as Hopf algebras if and only if $ (\alpha_{1}(N),\ast) \cong (\alpha_{2}(N),\ast) $ as $ G $-groups. For $ i=1,2 $ define an action $ \ast_{i} $ of $ G $ on $ N $ by $ g \ast_{i} \eta = \overline{\beta_{i}}(g) [ \eta ] $; then by \cite[(7.7) Proposition]{Ch00} we have that $ (\alpha_{i}(N), \ast) \cong (N, \ast_{i}) $ as $ G $-groups, and so $ (\alpha_{1}(N),\ast) \cong (\alpha_{2}(N),\ast) $ as $ G $-groups if and only if $ (N, \ast_{1}) \cong (N, \ast_{2}) $ as $ G $-groups. This occurs if and only if there exists $ \mu \in \Aut{N} $ such that 
\[ \mu(g \ast_{1} \eta) = g \ast_{2} \mu(\eta) \mbox{ for all } g \in G, \eta \in N,\]
that is, if and only if 
\[ \mu \overline{\beta_{1}}(g) = \overline{\beta_{2}}(g) \mu \mbox{ for all } g \in G. \]
\end{proof}

As a special case, we have

\begin{corollary} \label{corollary_Hopf_algebra_isomorphism_Aut_N_abelian}
If $ \Aut{N} $ is abelian (in particular, if $ N $ is cyclic), then $ E[\alpha_{1}(N)]^{G} \cong E[\alpha_{2}(N)]^{G} $ as Hopf algebras if and only if  $ \overline{\beta_{1}}(g) = \overline{\beta_{2}}(g) $ for all $ g \in G $. 
\end{corollary}

\begin{example}{\bf (The classical and canonical nonclassical Hopf-Galois structures revisited)}
Recall from Example \ref {example_Hopf_isomorphism_classical_nonclassical} that if $ L/K $ is a Galois extension with nonabelian group $ G $ then $ L/K $ admits at least two Hopf-Galois structures: the classical Hopf-Galois structure, which corresponds to the subgroup $ \rho(G) $ of $ \perm{G} $ and has Hopf algebra $ K[G] $, and the canonical nonclassical Hopf-Galois structure, which corresponds to the subgroup $ \lambda(G) $ of $ \perm{G} $ and has Hopf algebra $ H_{\lambda} = L[\lambda(G)]^{G} $. In the notation of this subsection, we may take $ N=G $ and view $ \lambda $ and $ \rho $ as embeddings of the abstract group $ G $ into $ \perm{G} $ whose images are $ G $-stable and regular. By following the details of the proof of Byott's translation theorem, we find that the embedding $ G \hookrightarrow \Hol{G} $ corresponding to $ \rho $ is $ \rho $ itself, and similarly for $ \lambda $. When we view $ \Hol{G} $ as $ \rho(G) \rtimes \Aut{G} $, we have $ \rho(G) = \{ (\rho(g),1) \mid g \in G \} $, whereas $ \lambda(G) = \{ (\rho(g^{-1}),c(g)) \mid g \in G \} $, where $ c(g) $ is the inner automorphism of $ G $ arising from conjugation by $ g $. Therefore $ \overline{\rho}(G) $ and $ \overline{\lambda}(G) $ have different orders, and so there cannot exist an automorphism $ \mu \in \Aut{G} $ with the properties required by Theorem \ref{theorem_Hopf_algebra_isomorphism_holomorph}. Hence we recover the fact that $ K[G] \not \cong H_{\lambda} $ as $ K $-Hopf algebras. 
\end{example}

\section{Algebra Isomorphisms for Commutative Structures} \label{section_algebra_isomorphisms_commutative}

In this section, we consider the question of when two Hopf algebras  $ E[N_{1}]^{G} $ and $ E[N_{2}]^{G} $ giving Hopf-Galois structures on $ L/K $ are isomorphic as $ K $-algebras. We shall assume that $ E[N_{1}]^{G} $ and $ E[N_{2}]^{G} $ are commutative algebras; this is equivalent to assuming that the underlying groups $ N_{1},N_{2} $ are abelian. Note, however, that we do not assume that these groups are isomorphic. We shall also assume that $ K $ has characteristic zero;  a consequence of this is that each $ E[N_{i}]^{G} $ is a semisimple $ K $-algebra and therefore by the Wedderburn-Artin theorem is isomorphic as a $ K $-algebra to a product of extension fields of $ K $. We establish a criterion for  $ E[N_{1}]^{G} \cong E[N_{2}]^{G} $ as $ K $-algebras. 
\\ \\
The question of explicitly determining the Wedderburn-Artin decomposition of a commutative Hopf algebra of this form has been answered by Boltje and Bley \cite{BleyBoltje} in terms of the action of the absolute Galois group of $ K $ on the values of characters of the underlying group. We review their approach using a finite extension of $ E $ in place of the separable closure of $ K $. Let $ e $ denote the least common multiple of the exponents of $ N_{1}, N_{2} $, let $ \zeta $ be a primitive $ e^{th} $ root of unity. Then $ K(\zeta) $ is a Galois extension of $ K $, and the compositum of $ E $ and $ K(\zeta) $, say $ \tilde{E} $, is a Galois extension of $ K $. Let $ \Gamma = \Gal{\tilde{E}/K} $ and $ \Gamma_{E} = \Gal{\tilde{E}/E} $, so that $ \Gamma / \Gamma_{E} \cong G $. 
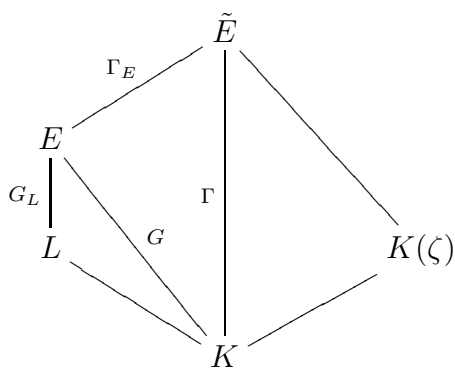
\[
\xymatrixcolsep{4pc} 
\xymatrixrowsep{2pc}
\xymatrix{
&\tilde{E} \ar@{-}[ddd]_{\Gamma}  & \\
E \ar@{-}[d]_{G_{L}} \ar@{-}[ddr]^{G} \ar@{-}[ur]^{\Gamma_{E}}& &  \\
L \ar@{-}[dr]& & \ar@{-}[dl] \ar@{-}[uul]K(\zeta)\\
&K&  \\
} 
\]
For $ i=1,2 $ let $ \Gamma $ act on $ \tilde{E}[N_{i}] $ by acting on $ \tilde{E} $ as Galois automorphisms and on $ N_{i} $ by factoring through $ \Gamma / \Gamma_{E} $ (i.e. through $ G $). Then
\[ \tilde{E}[N_{i}]^{\Gamma} = \left( \tilde{E}[N_{i}]^{\Gamma_{E}} \right)^{\Gamma / \Gamma_{E}} = \left( \tilde{E}^{\Gamma_{E}}[N_{i}] \right)^{\Gamma / \Gamma_{E}} = E[N_{i}]^{G}. \]
(In the second term $ \Gamma_{E} $ acts trivially on $ N_{i} $, so we have $ \tilde{E}[N_{i}]^{\Gamma_{E}}  = \tilde{E}^{\Gamma_{E}}[N_{i}] = E[N_{i}] $, the group algebra of $ N_{i} $ with coefficients drawn from $ \tilde{E}^{\Gamma_{E}} = E $.) 
Now elements of $ \widehat{N_{i}} $, namely the characters of $ N_{i} $, have values which all lie in $ \tilde{E} $, and so $ \tilde{E}[N_{i}] $ is isomorphic to $ \tilde{E}^{n} $ as an $ \tilde{E} $-algebra via orthogonal idempotents, one corresponding to each character. Writing $ \widehat{N_{i}} $ for the dual group of $ N_{i} $, the idempotent corresponding to $ \chi \in \widehat{N_{i}} $ is 
\[ e_{\chi} = \sum_{\eta \in N_{i}} \chi(\eta) \eta \in \tilde{E}[N_{i}]. \]
This action of $ \Gamma $ on $ \tilde{E}[N_{i}] $ permutes the orthogonal idempotents, which in turn induces an action of $ \Gamma $ on $ \widehat{N_{i}} $, as follows: if $ \chi \in \widehat{N_{i}} $ and $ \gamma \in \Gamma $ then 
\begin{eqnarray*}
\gamma \ast e_{\chi} & = & \gamma \ast \left( \sum_{\eta \in N_{i}} \chi(\eta) \eta \right) \\
& = & \sum_{\eta \in N_{i}} \gamma(\chi(\eta)) \gamma \ast \eta \\
& = & \sum_{\eta \in N_{i}} \gamma(\chi(\gamma^{-1} \ast \eta)) \eta \mbox{ (reindexing).}
\end{eqnarray*}
We therefore define $ \gamma \ast \chi $ by $ (\gamma \ast \chi [\eta] = \gamma(\chi[\gamma^{-1} \ast \eta])  $ for all $ \eta \in N_{i} $. Now let $ \chi_{1}, \ldots ,\chi_{r} $ be representatives of the orbits of $ \Gamma $ in $ \widehat{N_{i}} $, and for each $ j=1, \ldots ,r $ let $ S_{j} = \mbox{Stab}(\chi_{j}) $. Then \cite[Lemma 2.2]{BleyBoltje} asserts that
\[ \tilde{E}[N_{i}]^{\Gamma} \cong \prod_{j=1}^{r} \tilde{E}^{S_{j}} \mbox{ as $ K $-algebras.} \]
From this, it is clear that if there exist sets of representatives of the orbits of $ \Gamma $ in $ \widehat{N_{1}} $ and $ \widehat{N_{2}} $ having the same stabilizers then $ \tilde{E}[N_{1}]^{\Gamma} \cong \tilde{E}[N_{2}]^{\Gamma} $ as $ K $-algebras. The following theorem is a refinement of this idea. Note that  as in section \ref{section_Hopf_algebra_isomorphisms} we formulate the theorem in more general terms than we require: the added flexibility will be useful later. 

\begin{theorem} \label{theorem_algebra_isomorphism}
Let $ (N_{1}, \ast_{1}) $ and $ (N_{2}, \ast_{2}) $ be abelian $ \Gamma $-groups and, for $ i=1,2 $, let $ \Gamma $ act on $ \tilde{E}[N_{i}] $ by acting on $ \tilde{E} $ as Galois automorphisms and on $ N_{i} $ via $ \ast_{i} $. Then $ \tilde{E}[N_{1}]^{\Gamma} \cong  \tilde{E}[N_{2}]^{\Gamma} $ as $ K $-algebras if and only if $ \widehat{N_{1}} \cong \widehat{N_{2}} $ as $ \Gamma $-sets.  
\end{theorem}
\begin{proof}
Suppose first that $ \psi  : \widehat{N_{1}} \rightarrow \widehat{N_{2}} $ is a $ \Gamma $-equivariant bijection. Let $ \chi_{1}, \ldots , \chi_{r} $ be representatives of the orbits of $ \Gamma $ in $ \widehat{N_{1}} $. Then $ \psi(\chi_{1}), \ldots , \psi(\chi_{r}) $ are representatives of the orbits of $ \Gamma $ in $ \widehat{N_{2}} $, and for each $ i = 1, \ldots ,r $ we have
\begin{eqnarray*}
\gamma \in \mbox{Stab}(\psi(\chi_{i})) & \Leftrightarrow & \gamma \ast \psi(\chi_{i}) = \psi(\chi_{i}) \\
& \Leftrightarrow & \psi(\gamma \ast \chi_{i}) = \psi(\chi_{i}) \mbox{ ($\psi$ is $ \Gamma $-equivariant)}\\
& \Leftrightarrow &\gamma \ast \chi_{i} = \chi_{i} \mbox{ ($\psi$ is a bijection)}\\
& \Leftrightarrow & \gamma \in \mbox{Stab}(\chi_{i}). 
\end{eqnarray*}
Therefore $ \mbox{Stab}(\psi(\chi_{i})) = \mbox{Stab}(\chi_{i}) $ for each $ i $, and so 
\[ H_{1} \cong \prod_{i=1}^{r} \tilde{E}^{\text{Stab}(\chi_{i})} = \prod_{i=1}^{r} \tilde{E}^{\text{Stab}(\psi(\chi_{i}))} \cong H_{2} \mbox{ as $ K $-algebras}. \]
Conversely, suppose that $ \varphi : H_{1} \xrightarrow{\sim} H_{2} $ is an isomorphism of $ K $-algebras. It extends to an isomorphism of $ \tilde{E} $-algebras $ \varphi : \tilde{E}[N_{1}] \cong \tilde{E}[N_{2}] $ which, by the argument employed in the proof of Theorem \ref{theorem_Hopf_algebra_isomorphism}, is $ \Gamma $-equivariant. 
Now recall that $ \tilde{E}[N_{1}] \cong \tilde{E}[N_{2}]  \cong \tilde{E}^{n} $ as $ \tilde{E} $-algebras via orthogonal idempotents, as described above. Let $ \{ e_{\chi} \mid \chi \in \widehat{N_{1}} \} $ be the orthogonal idempotents of $ \tilde{E}[N_{1}] $, and $ \{ f_{\chi^{\prime}} \mid \chi^{\prime} \in \widehat{N_{2}} \} $ be the orthogonal idempotents of $ \tilde{E}[N_{2}] $. The set $ \{ \varphi(e_{\chi}) \mid \chi \in \widehat{N_{1}} \} $ is the $ \tilde{E} $-basis of orthogonal idempotents of $ \tilde{E}[N_{2}] $, and so for each $  \chi \in \widehat{N_{1}} $ there exists $ \psi(\chi) \in \widehat{N_{2}} $ such that $ \varphi(e_{\chi}) = f_{\psi(\chi)} $. This establishes a bijection $ \psi : \widehat{N_{1}} \rightarrow \widehat{N_{2}} $. Since $ \varphi $ is $ \Gamma $-equivariant, for $ \gamma \in \Gamma $ we have:
\[ f_{(\gamma \ast\psi)(\chi)} = \gamma \ast f_{\psi(\chi)} = \gamma \ast \varphi( e_{\chi} ) =  \varphi(\gamma \ast e_{\chi} ) = \varphi( e_{\gamma \ast \chi} ) = f_{\psi(\gamma \ast \chi)}, \]
and so $ (\gamma \ast \psi)(\chi) = \psi(\gamma \ast \chi) $. Therefore the bijection $ \psi : \widehat{N_{1}} \rightarrow \widehat{N_{2}} $ is $ \Gamma $-equivariant. 
\end{proof}

The following corollary is the principal application of Theorem \ref{theorem_algebra_isomorphism}. 

\begin{corollary} \label{corollary_algebra_isomorphism}
Let $ E[N_{1}]^{G} $ and $ E[N_{2}]^{G} $ be commutative Hopf algebras giving Hopf-Galois structures on $ L/K $. Then $ E[N_{1}]^{G} \cong E[N_{2}]^{G} $ as $ K $-algebras if and only if $ \widehat{N_{1}} \cong \widehat{N_{2}} $ as $ \Gamma $-sets. 
\end{corollary}

\begin{example}{\bf (Elementary abelian extensions of degree $ p^{2} $ revisited) }
Let $ p $ be an odd prime and let $ L/K $ be an elementary abelian extension of degree $ p^{2} $ with group $ G $. In Example \ref{example_Hopf_isomorphism_elementary_abelian_p2} we determined criteria for two Hopf algebras $ H_{1}, H_{2} $ giving Hopf-Galois structures on $ L/K $ to be isomorphic as $ K $-Hopf algebras. Under the additional hypotheses that $ K $ has characteristic zero and contains a primitive $ p^{th} $ root of unity $ \zeta $, we now apply Corollary \ref{corollary_algebra_isomorphism} to determine criteria for them to be isomorphic as $ K $-algebras. In fact, we show that in this case $ H_{1} \cong H_{2} $ as $ K $-algebras if and only if $ H_{1} \cong H_{2} $ as $ K $-Hopf algebras. 
\\ \\
Recall from Example \ref{example_Hopf_isomorphism_elementary_abelian_p2} a Hopf algebra $ H $ giving a Hopf-Galois structures on $ L/K $ determined by a choice of subgroup $ T $ of degree $ p $ and an integer $ d \in \{0, \ldots ,p-1\} $; the Hopf algebra is then $ L[N]^{G} $, where $ N $ is the subgroup of $ \perm{G} $ generated by two permutations $ \alpha, \beta $ as described in Equations \eqref{equation_byott_p2_permutations}, and the action of $ G $ on $ N $ is as described in Equation \eqref{equation_byott_p2_lambda_action}. The dual group $ \widehat{N} $ is therefore generated by two characters $ \chi, \psi $, defined as follows:
\[ \chi(\alpha) = \zeta , \hspace{4mm} \chi(\beta) = 1 , \hspace{4mm} \psi(\alpha) = 1 , \hspace{4mm} \psi(\beta) = \zeta. \]
Since $ L/K $ is a Galois extension we have $ E = L $, and the hypothesis that $ K $ contains a primitive $ p^{th} $ root of unity implies that in the notation of this section we have $ \tilde{E}=E=L $ and $ \Gamma = G $. The action of $ G $ on $ N $ described in Equation \eqref{equation_byott_p2_lambda_action} translates into an action of $ G $ on $ \widehat{N} $ by 
\begin{equation} \label{equation_byott_p2_lambda_action_dual}
 s \ast \chi = \chi \psi^{d^{-1}}, \hspace{4mm} t \ast \chi = \chi, \hspace{4mm} s \ast \psi = \psi, \hspace{4mm} t \ast \psi=\psi,
\end{equation}
where $ d^{-1} $ is to be interpreted modulo $ p $. Now let $ H_{1}, H_{2} $ be two such Hopf algebras giving distinct Hopf-Galois structures on $ L/K $, with underlying groups $ N_{1}, N_{2} $ as in Example \ref{example_Hopf_isomorphism_elementary_abelian_p2}. We claim that $ H_{1} \cong H_{2} $ as $ K $-algebras if and only if $ d_{1}=d_{2}=0 $ or $ d_{1}d_{2} \neq 0 $ and $ T_{1}=T_{2} $; these are identical to the conditions we derived in Example \ref{example_Hopf_isomorphism_elementary_abelian_p2} for $ H_{1} \cong H_{2} $ as $ K $-Hopf algebras. If these conditions are satisfied then $ H_{1} \cong H_{2} $ as $ K $-Hopf algebras, so certainly as $ K $-algebras. For the converse, note that by \eqref{equation_byott_p2_lambda_action_dual} if $ d_{i}=0 $ then $ G $ acts trivially on $ \widehat{N_{i}} $, whereas if $ T $ is a subgroup of $ G $ of order $ p $ and $ d \neq 0 $ then the kernel of the action of $ G $ on $ N_{T,d} $ is precisely $ T $. Therefore if $ \widehat{N_{1}} \cong \widehat{N_{2}} $ as $ G $-sets then we must have $ d_{1} = d_{2} = 0 $ or $ d_{1}d_{2} \neq 0 $ and $ T_{1} = T_{2} $. 
\\ \\
We remark that if one determines a set of representatives for the orbits of $ \Gamma $ in $ \widehat{N_{i}} $, and the associated stabilizers, then one finds that
$ H_{i} \cong K^{p} \times \left(L^{T_{i}}\right)^{p-1} $ as $ K $-algebras; this result was obtained by slightly different methods in \cite[Proposition 3.4]{PJT_CpxCp}. 
\end{example}

\subsection{Algebra isomorphisms via the holomorph}

Hitherto in this section we have not assumed that $ N_{1}, N_{2} $ are isomorphic as groups; we now impose this assumption. We can therefore view $ N_{1}, N_{2} $ as the images of a single abstract abelian group $ N $ under two embeddings $ \alpha_{1}, \alpha_{2} : N \hookrightarrow \perm{X} $, as in subsection \ref{subsection_Hopf_algebra_isomorphisms_holomorph}. We recall from that subsection that By Byott's translation theorem these embeddings correspond to embeddings $ \beta_{1}, \beta_{2} : G \hookrightarrow \Hol{N} = \rho(N) \rtimes \Aut{N} $, and we write $ \overline{\beta_{1}}, \overline{\beta_{2}} $ for the compositions of $ \beta_{1}, \beta_{2} $ with the projection onto the $ \Aut{N} $ component. We have seen in Theorem \ref{theorem_Hopf_algebra_isomorphism_holomorph} that it is possible to detect when $ E[\alpha_{1}(N)]^{G} \cong E[\alpha_{2}(N)]^{G} $ as $ K $-Hopf algebras by studying properties of $ \overline{\beta_{1}} , \overline{\beta_{2}} $. We shall show that in our situation these maps also allow us to detect $ K $-algebra isomorphisms.
\\ \\
As in the proof of Theorem \ref{theorem_Hopf_algebra_isomorphism_holomorph}, for $ i=1,2 $ define an action $ \ast_{i} $ of $ G $ on $ N $ by $ g \ast_{i} \eta = \overline{\beta_{i}}(g) [ \eta ] $; then by \cite[(7.7)]{Ch00} we have that $ (\alpha_{i}(N), \ast) \cong (N, \ast_{i}) $ as $ G $-groups. We may extend these to actions of $ \Gamma $ by factoring through $ G $, obtaining $ (\alpha_{i}(N), \ast) \cong (N, \ast_{i}) $ as $ \Gamma $-groups.  Each of the actions of $ \Gamma $ on $ N $ yields a dual action of $ \Gamma $ on $ \widehat{N} $, which we also denote by $ \ast_{i} $. Similarly, the action of $ \Gamma $ on each $ \alpha_{i}(N) $ yields a dual action $ \ast $ of $ \Gamma $ on each $ \widehat{\alpha_{i}(N)} $. Then we have:

\begin{lemma} \label{lemma_dual_equivairant_isomorphism}
For $ i=1,2 $, we have $ (\widehat{\alpha_{i}(N)}, \ast) \cong (\widehat{N}, \ast_{i}) $ as $ \Gamma $-groups. 
\end{lemma} 
\begin{proof}
For each $ i=1,2 $, the map $ \alpha_{i} : (N,\ast_{i}) \rightarrow (\alpha_{i}(N),\ast) $ is a $ \Gamma $-equivariant isomorphism. For each $ \chi \in \widehat{N} $, the function $ \psi_{\chi} : \alpha_{i}(N) \rightarrow \tilde{E} $ defined by $ \psi_{\chi}[\alpha_{i}(\eta)] = \chi[\eta] $ for all $ \eta \in N $ is actually a character of $ \alpha_{i}(N) $, since for $ \eta,\eta^{\prime} \in N $ we have
\begin{eqnarray*}
\psi_{\chi}[\alpha_{i}(\eta)\alpha_{i}(\eta^{\prime})] & = & \psi_{\chi}[\alpha_{i}(\eta\eta^{\prime})] \\
& = & \chi[\eta\eta^{\prime}] \\
& = & \chi[\eta]\chi[\eta^{\prime}] \\
& = & \psi_{\chi}[\alpha_{i}(\eta)]\psi_{\chi}[\alpha_{i}(\eta^{\prime})].
\end{eqnarray*}
We shall show that the map $ \psi : \widehat{N} \rightarrow \widehat{\alpha_{i}(N)} $ defined by $ \psi(\chi) = \psi_{\chi} $ is a $ \Gamma $-equivariant isomorphism. To show that it is an injection, let $ \chi, \chi^{\prime} \in \widehat{N} $. Then
\begin{eqnarray*}
\psi(\chi) = \psi(\chi^{\prime}) & \Rightarrow & \psi(\chi)[\alpha_{i}(\eta)] = \psi(\chi^{\prime})[\alpha_{i}(\eta)] \mbox{ for all $ \eta \in N $} \\
& \Rightarrow & \chi(\eta) = \chi^{\prime}(\eta) \mbox{ for all $ \eta \in N $}\\
& \Rightarrow & \chi = \chi^{\prime}. 
\end{eqnarray*}
Thus $ \psi $ is an injection and, since $ |\widehat{N}| = |\widehat{\alpha_{i}(N)}| $, therefore a bijection. Next we show that $ \psi $ is a homomorphism. Let $ \chi, \chi^{\prime} $ be as above; then for all $ \eta \in N $ we have
\begin{eqnarray*}
\psi(\chi \chi^{\prime})[\alpha_{i}(\eta)] & = & (\chi \chi^{\prime})[\eta] \\
& = & \chi[\eta] \chi^{\prime}[\eta] \\
& = & \psi(\chi)[\alpha_{i}(\eta)] \psi(\chi^{\prime})[\alpha_{i}(\eta)] \\
& = & (\psi(\chi)\psi(\chi^{\prime}))[\alpha_{i}(\eta)].
\end{eqnarray*}
Hence $ \psi(\chi \chi^{\prime}) = \psi(\chi)\psi(\chi^{\prime}) $, and so $ \psi $ is a homomorphism. Finally, we show that $ \psi $ is $ \Gamma $-equivariant. Let $ \chi \in \widehat{N} $ and $ \gamma \in \Gamma $; then for all $ \eta \in N $ we have
\begin{eqnarray*}
(\gamma \ast \psi(\chi))[\alpha_{i}(\eta)] & = & \gamma(\psi(\chi)[\gamma^{-1} \ast \alpha_{i}(\eta)]) \\
& = & \gamma(\psi(\chi)[\alpha_{i}(\gamma^{-1} \ast_{i} \eta)])  \\
& = & \gamma(\chi[\gamma^{-1} \ast_{i} \eta]) \\
& = & (\gamma \ast_{i} \chi)[\eta] \\
& = & (\psi(\gamma \ast \chi))[\alpha_{i}(\eta)]. 
\end{eqnarray*}
Therefore $ \gamma \ast \psi(\chi) = \psi(\gamma \ast \chi) $, and so $ \psi $ is $ \Gamma $-equivariant.  
\\ \\
We have shown that $ \psi : \widehat{N} \rightarrow \widehat{\alpha_{i}(N)} $ is a $ \Gamma $-equivariant isomorphism, and so $ (\widehat{\alpha_{i}(N)}, \ast) \cong (\widehat{N}, \ast_{i}) $ as $ \Gamma $-groups. 
\end{proof}

\begin{theorem} \label{theorem_algebra_isomorphism_holomorph}
We have $ E[\alpha_{1}(N)]^{G} \cong E[\alpha_{2}(N)]^{G} $ as $ K $-algebras if and only if $ (\widehat{N},\ast_{1}) \cong (\widehat{N},\ast_{2}) $ as $ \Gamma $-sets. 
\end{theorem}
\begin{proof}
By Theorem \ref{theorem_algebra_isomorphism} we have $ \tilde{E}[\alpha_{1}(N)]^{\Gamma} \cong  \tilde{E}[\alpha_{2}(N)]^{\Gamma} $ as $ K $-algebras if and only if $ (\widehat{\alpha_{1}(N)},\ast) \cong (\widehat{\alpha_{2}(N)},\ast) $ as $ \Gamma $-sets, and by Lemma \ref{lemma_dual_equivairant_isomorphism} this occurs if and only if $ (\widehat{N},\ast_{1}) \cong (\widehat{N},\ast_{2}) $ as $ \Gamma $-sets.
\end{proof}

The following corollary shows that for certain structures of cyclic type we can dispense with $ \widehat{N} $ and work directly with $ N $: 

\begin{corollary} \label{corollary_algebra_isomorphism_cyclic}
Suppose that $ N $ is cyclic and that $ K $ contains a primitive $ |N|^{th} $ root of unity $ \zeta $. Then the following are equivalent:
\begin{enumerate}
\item $ E[\alpha_{1}(N)]^{G} \cong E[\alpha_{2}(N)]^{G} $ as $ K $-algebras;
\item $ (\widehat{\alpha_{1}(N)},\ast) \cong (\widehat{\alpha_{2}(N)},\ast) $ as $ \Gamma $-sets. 
\item $ (\widehat{N},\ast_{1}) \cong (\widehat{N},\ast_{2}) $ as $ \Gamma $-sets;
\item $ (\alpha_{1}(N),\ast) \cong (\alpha_{2}(N),\ast) $ as $ \Gamma $-sets. 
\item $ (N,\ast_{1}) \cong (N,\ast_{2}) $ as $ \Gamma $-sets;
\end{enumerate}
\end{corollary}
\begin{proof}
(1), (2), and (3) are equivalent by Theorem \ref{theorem_algebra_isomorphism} and Theorem \ref{theorem_algebra_isomorphism_holomorph}, and (4) and (5) are equivalent since for each $ i $ we have $ (\alpha_{i}(N), \ast) \cong (N, \ast_{i}) $ as $ \Gamma $-groups. We show that (3) is equivalent to (5). 
\\ \\
Let $ \eta $ be a generator of $ N $, and define $ \chi : N \rightarrow K $ by $ \chi(\eta) = \zeta $; then $ \chi $ is a generator of $ \widehat{N} $. We claim that for $ i=1,2 $, the isomorphism  $ f : N \rightarrow \widehat{N} $ defined by $ f(\eta) = \chi $ has the property that 
\[ (\gamma \ast_{i} f)[\eta] = f(\gamma^{-1} \ast_{i} \eta). \]
To see this, note that each $ \gamma \in \Gamma $ acts as an automorphism of $ N $, so there exists an integer $ e(i,\gamma) $ (coprime to $ |N|$) such that $ \gamma \ast_{i} \eta = \eta^{e(i,\gamma)} $. Now using the assumption that $ \zeta \in K $, we have:
\[ (\gamma \ast_{i} \chi) [\eta] = \chi[\gamma^{-1} \ast_{i} \eta] = \chi[\eta^{e(i,\gamma)^{-1}}] = \chi[\eta]^{e(i,\gamma)^{-1}} = \chi^{e(i,\gamma)^{-1}}[\eta], \]
so in fact $ \gamma \ast_{i} \chi = \chi^{e(i,\gamma)^{-1}} $ (where $ e(i,\gamma)^{-1} $ is computed modulo $ |N| $). It follows that the isomorphism $ f $ has the desired property. We therefore obtain a diagram:
\[
\xymatrixcolsep{4pc} 
\xymatrixrowsep{4pc}
\xymatrix{ 
(N,\ast_{1}) \ar[d]^{f} \ar@{-->}[r]^{\pi} & (N,\ast_{2}) \ar[d]^{f} \\
(\widehat{N},\ast_{1}) \ar@{-->}[r]^{\widehat{\pi}} & (\widehat{N},\ast_{2}) }
\]
If $ \pi $ is a $ \Gamma $-equivariant bijection, then $ f \circ \pi \circ f^{-1} $ is a $ \Gamma $ equivariant bijection, and if $ \widehat{\pi} $ is a $ \Gamma $-equivariant bijection, then $ f^{-1} \circ \psi \circ f $ is a $ \Gamma $-equivariant bijection. 
\end{proof}

\section{Cyclic Extensions of Prime Power Degree} \label{section_cyclic_p_power_extensions}

Let $ p $ be an odd prime number and $ L/K $ a cyclic extension of degree $ p^{n} $. By a result of Kohl \cite[Theorem 3.3]{Kohl1998} (see also \cite[(9.1)]{Ch00}), there are precisely $ p^{n-1} $ Hopf-Galois structures on $ L/K $, and they all have cyclic type. Explicit $ K $-algebra generators for the Hopf algebras appearing in these Hopf-Galois structures were determined in \cite[\S 6.3]{Childs2011}, requiring intricate manipulations. In this section we apply to results of section \ref{section_Hopf_algebra_isomorphisms} and \ref{section_algebra_isomorphisms_commutative} to determine which of the Hopf algebras appearing in these Hopf-Galois structures are isomorphic as $ K $-Hopf algebras or $ K $-algebras, and (under certain additional hypotheses) explicitly determine their Wedderburn-Artin decompositions. 
\\ \\
Since the Hopf-Galois structures admitted by $ L/K $ all have cyclic type, we can view the corresponding regular subgroups of $ \perm{G} $ as images of a single abstract cyclic group $ N = \langle \eta \rangle $ of order $ p^{n} $ under $ p^{n-1} $ different regular embeddings $ \alpha_{s} : N \hookrightarrow \perm{G} $. By Byott's translation, each such $ \alpha_{s} $ corresponds to an embedding $ \beta_{s} : G \rightarrow \Hol{N} $, and these are described in \cite[(8.6) and (9.1)]{Ch00}: let $ G = \langle \sigma \rangle $, and let $ \delta $ be the $ (p-1)^{st} $ power of some generator of the cyclic group $ \Aut{N} $. Then the embeddings we seek are of the form $ \beta_{s}: G \hookrightarrow \Hol{N} $ with
\[ \beta_{s}(\sigma) = ( \rho(\eta), \delta^{s} ), 0 \leq s < p^{n-1}. \]
For each $ s $, let $ \alpha_{s} : N \hookrightarrow \perm{G} $ denote the embedding corresponding to $ \beta_{s} $, and let $ H_{s} = L[\alpha_{s}(N)]^{G} $ denote the corresponding Hopf algebra. 

\begin{theorem} \label{theorem_cyclic_p_power_Hopf_algebra_isomorphisms}
Let $ 0 \leq r,s < p^{n-1} $. Then $ H_{r} \cong H_{s} $ as $ K $-Hopf algebras if and only if $ r = s $.
\end{theorem}
\begin{proof}
Since $ N $ is cyclic, $ \Aut{N} $ is abelian, and so by Corollary \ref{corollary_Hopf_algebra_isomorphism_Aut_N_abelian} we have $ H_{r} \cong H_{s} $ as $ K $-Hopf algebras if and only if $ \overline{\beta_{r}}(g) = \overline{\beta_{s}}(g) $ for all $ g \in G $. Since in this case $ G $ is generated by $ \sigma $, this occurs if and only if $ \overline{\beta_{r}}(\sigma) = \overline{\beta_{s}}(\sigma) $; that is, if and only if $ \delta^{r} = \delta^{s} $. Hence $ H_{r} \cong H_{s} $ as $ K $-Hopf algebras if and only if $ r = s $.
\end{proof}

Therefore the Hopf algebras giving the Hopf-Galois structures on $ L/K $ are pairwise nonisomorphic. We can use Theorem \ref{theorem_Hopf_algebra_isomorphism_base_change} to determine which of them become isomorphic under various base changes. Let 
\[ K = K_{0} \subset K_{1} \subset \cdots \subset K_{n}=L \]
be the maximal tower of field extensions, and for each $ i=0, \ldots ,n $ let $ G_{i} =\langle \sigma^{p^{i}} \rangle = \Gal{L/K_{i}} $.

\begin{theorem} \label{theorem_cyclic_p_n_Hopf_isomorphism}
For $ 0 \leq r,s \leq p^{n-1} $ and $ 0 \leq i \leq n $, we have $ K_{i} \otimes_{K} H_{r} \cong K_{i} \otimes_{K} H_{s} $ as $ K_{i} $-Hopf algebras if and only if $ r \equiv s \pmod {p^{n-1-i}} $. 
\end{theorem}
\begin{proof}
By Theorem \ref{theorem_Hopf_algebra_isomorphism_base_change}, we have $ K_{i} \otimes_{K} H_{r} \cong K_{i} \otimes_{K} H_{s} $ as $ K_{i} $-Hopf algebras if and only if $ (\alpha_{r}(N),\ast) \cong (\alpha_{s}(N), \ast) $ as $ G_{i} $-groups. By \cite[(7.7)]{Ch00}, this is equivalent to $ (N, \ast_{r}) \cong (N, \ast_{s}) $ as $ G_{i} $-groups, so we must show that this occurs if and only if $ r \equiv s \pmod {p^{n-1-i}} $. 
\\ \\
Recall that $ \delta \in \Aut{N} $ has order $ p^{n-1} $, so there exists an element $ d \in (\mathbb{Z}/p^{n}\mathbb{Z})^{\times} $ of order $ p^{n-1} $ such that $ \delta(\eta)=\eta^{d} $. It follows that for $ 0 \leq j \leq p^{n}-1 $ we have
\[ \sigma^{j} \ast_{r} \eta = \overline{\beta_{r}}(\sigma)^{j} [\eta] = \delta^{rj} \eta = \eta^{d^{rj}}, \]
and similarly $ \sigma^{j} \ast_{s} \eta = \eta^{d^{sj}} $. Now let $ \theta $ be an automorphism of $ N $, and write $ \theta(\eta) = \eta^{t} $ for some integer $ t $ coprime to $ p $. Then for $ 0 \leq i \leq n $ we have:
\[ \begin{array}{cccccc}
&\theta \left( \sigma^{p^{i}} \ast_{r} \eta \right) & = & \theta \left( \eta^{d^{rp^{i}}} \right) & = &  \eta^{td^{rp^{i}}} \\
\mbox{ and }&\sigma^{p^{i}} \ast_{s} \theta \left( \eta \right) & = & \sigma^{p^{i}} \ast_{s} \eta^{t} & = & \eta^{td^{sp^{i}}},
\end{array} \]
so $ \theta $ is $ G_{i} $-equivariant if and only if $ d^{rp^{i}} \equiv d^{sp^{i}} \pmod{p^{n}} $. Since $ d $ has order $ p^{n-1} $ in $ (\mathbb{Z}/p^{n}\mathbb{Z})^{\times}  $, this occurs if and only if $ rp^{i} \equiv sp^{i} \pmod{p^{n-1}} $, that is, if and only if $ r \equiv s \pmod{p^{n-1-i}} $.
\end{proof}

\begin{corollary}
Let $ 0 \leq i \leq n $. Then:
\begin{enumerate}
\item The collection $ \{ K_{i} \otimes_{K} H_{0}, \ldots ,K_{i} \otimes_{K} H_{p^{n-1}-1} \} $ can be partitioned into $ p^{n-1-i} $ Hopf algebra isomorphism classes;
\item Each class contains $ p^{i} $ Hopf algebras;
\item $ \{ K_{i} \otimes H_{0}, \ldots ,K_{i} \otimes_{K} H_{p^{n-2-i}} \} $ is a complete set of representatives for the classes;
\item For $ 0 \leq j < p^{n-1} $, the class containing $ K_{i} \otimes_{K} H_{j} $ is 
\[ \{ K_{i} \otimes_{K} H_{j+p^{n-i}m} \mid 0  \leq m < p^{i} \}. \]
\end{enumerate}
\end{corollary}

\begin{corollary}
Let $ 0 \leq r \leq p^{n}-1 $ and $ 0 \leq i \leq n $. Then $ K_{i} \otimes_{K} H_{r} \cong K_{i}[N] $ as $ K_{i} $-Hopf algebras if and only if $ r \equiv 0 \pmod{n-1-i} $ 
\end{corollary}
\begin{proof}
We have $ H_{0} \cong K[N] $ as $ K $-Hopf algebras, so for each $ i $ we have $ K_{i} \otimes_{K} H_{0} \cong K_{i}[N] $ as $ K_{i} $-Hopf algebras. The result now follows from Theorem \ref{theorem_cyclic_p_n_Hopf_isomorphism}.
\end{proof}

We now impose the additional assumption that $ K $ has characteristic zero and contains a primitive $ p^{n} $-root of unity $ \zeta $. In this case we can use Corollary \ref{corollary_algebra_isomorphism_cyclic} to determine which of the $ K $-Hopf algebras appearing in the classification of Hopf-Galois structures on $ L/K $ are isomorphic as $ K $-algebras. 

\begin{theorem} \label{theorem_cyclic_p_power_algebra_isomorphisms}
For $ 0 \leq r,s < p^{n-1} $, we have $ H_{r} \cong H_{s} $ as $ K $-algebras if and only if $ v_{p}(r)=v_{p}(s) $, where $ v_{p} $ denotes the $ p $-adic valuation function.
\end{theorem}
\begin{proof}
By Corollary \ref{corollary_algebra_isomorphism_cyclic}, we have $ H_{r} \cong H_{s} $ if and only if $ (N,\ast_{r}) \cong (N,\ast_{s}) $ as $ G $-sets, so we must show that this occurs if and only if $ v_{p}(r)=v_{p}(s) $. 
\\ \\
Suppose first that $ v_{p}(r)=v_{p}(s) $. Then since $ \Aut{N} $ is cyclic we must have $ \langle \delta^{r} \rangle = \langle \delta^{s} \rangle = \Delta $ say, and we have $ \overline{\beta_{r}}(G) = \overline{\beta_{s}}(G) $. Therefore for each $ \mu \in N $, the orbits of $ \mu $ with respect to $ \ast_{r} $ and $ \ast_{s} $ coincide, and so the stabilizers $ \text{Stab}_{r}(\mu),  \text{Stab}_{s}(\mu) $ of $ \mu $ with respect to $ \ast_{r} , \ast_{s}$ have the same order.  Since $ G $ is cyclic, this implies that they are equal. 
\\ \\
Now let $ \eta_{1}, \ldots, \eta_{k} $ be representatives for the orbits of $ (N,\ast_{r}) $, and define $ \pi : N \rightarrow N $ by setting $ \pi(\eta_{i})=\eta_{i} $ for each $ i $, and insisting that $ \pi( g \ast_{r} \mu ) = g \ast_{s} \pi(\mu) $ for all $ \mu \in N $. It is routine to verify that $ \pi $ is well defined and injective, and so it is a $ G $-equivariant bijection from $ (N,\ast_{r}) $ to $ (N,\ast_{s}) $. 
\\ \\
Conversely, suppose that $ v_{p}(r) \neq v_{p}(s) $, and assume without loss of generality that $ v_{p}(r) < v_{p}(s) $. Then since $ \Aut{N} $ is cyclic we have $ \overline{\beta_{s}}(G) \subsetneq \overline{\beta_{r}}(G) $, and so (for example) the orbit of $ \eta $ with respect to $ \ast_{s} $ is strictly contained in the orbit of $ \eta $ with respect to $ \ast_{r} $. Therefore $ (N,\ast_{r}) $ and $ (N,\ast_{s}) $ cannot be isomorphic $ G $-sets in this case. 
\end{proof}

\begin{corollary} \label{corollary_cyclic_p_power_algebra_isomorphisms}
Precisely $ n $ non-isomorphic $ K $-algebras appear in the classification of Hopf-Galois structures on $ L/K $. For each $ 0 \leq v \leq n-1 $, the $ K $-algebra $ H_{p^{v}} $ has $ \varphi(p^{v}) $ distinct Hopf-Galois actions on $ L/K $.
\end{corollary}

Finally, retaining the assumption that $ \zeta \in K $, we explicitly compute the Wedderburn-Artin decompositions of these algebras. Recall that, by \cite[Lemma 2.5]{BleyBoltje}, for $ 0 \leq r < p^{n-1} $ we have
\[ H_{r} = L[\alpha_{r}(N)]^{G}  \cong \prod_{m=1}^{t} L^{ S_{m} } \mbox{ as $ K $-algebras,} \]
where the $ S_{m} $ are the stabilizers of a set of representatives of the orbits of $ G $ in $ \widehat{\alpha_{r}(N)} $. By Corollary \ref{corollary_algebra_isomorphism_cyclic}, these stabilizers coincide with those of a set of representatives of the orbits of $ G $ in $ N $, with $ G $ acting by
\[ \sigma^{i} \ast_{r} \eta^{j} = \overline{\beta_{r}}(\sigma^{i})[\eta^{j}] = \delta^{ir}[\eta^{j}] = \eta^{jd^{ir}}. \]
Since $ N $ is cyclic, we may translate this to an action of the additive group $\mathbb{Z}/p^n\mathbb{Z}$ on itself via
\[i \cdot_{r} j = jd^{ir},\]
and study the orbits and stabilizers of this action. 

\begin{lemma}\label{lemma_cyclic_orbits} 
Let $ j \in \mathbb{Z}/p^n\mathbb{Z} $ and $ m=\max\{n-1-v_p(j)-v_p(r),0\} $. Then
\[ {\mathcal O}(j)=\{ jd^{is}: 0 \leq i < p^{m} \}, \mbox{ and } \mbox{Stab}(j)= \langle p^{m} \rangle.\]
\end{lemma}
\begin{proof}  Note $i\cdot_{r} j = j$ if and only if 
\[
 jd^{{ir}} \equiv j \pmod{p^n}
\]
i.e.,
\[d^{{ir}} \equiv 1 \pmod{p^{n-v_p(j)}}, \]
which holds if and only if
\[{ir} \equiv 0 \pmod{p^{n-v_p(j)-1}}.\]
Now if $m=0$ then $v_p(r)=n-1-v_p(j)$, hence $p^{n-1-v_p(j)}\mid r$ and the result is clear. Otherwise, the above congruence holds if and only if \[i\equiv 0 \pmod{p^{n-1-v_p(j)-v_p(r)}}.\]
Thus, $\mbox{Stab}(j)=\langle p^{n-1-v_p(j)-v_p(r)} \rangle $. The orbit computation follows immediately.
\end{proof}

The following allows us to count orbit classes in the cases $m>0$.

\begin{lemma} \label{lemma_cyclic_orbits_m>0}
Let $0 < r \leq p^{n-1}$ and pick, if possible, $0 < m \leq n-1-v_p(s)$. Then $\mathbb{Z}/p^n\mathbb{Z}$ has $p^{v_p(r)}(p-1)$ distinct orbits whose stabilizer is $ \langle p^{m} \rangle $.
\end{lemma}
\begin{proof} 
There are $\varphi(p^{m+v_p(r)+1})= p^{m+v_p(r)}(p-1)$ elements of order $p^{m+v_p(r)+1}$ in $\mathbb{Z}/p^n\mathbb{Z}$. Clearly, $j$ has order $p^{m+v_p(r)+1}$ if and only if $v_p(j)=n-m-v_p(r)-1$. Thus, there are $ p^{m+v_p(r)}(p-1)$ choices of $j$ for which $m=n-1-v_p(j)-v_p(r)$. By Lemma \ref{lemma_cyclic_orbits}  there are $p^m$ choices for $j$ in each orbit. Thus, the number of orbits whose stabilizer is $ \langle p^m \rangle $ is
\[ \frac{p^{m+v_p(r)}(p-1)}{p^m} = p^{v_p(r)}(p-1).\]
\end{proof}

For $m=0$ we have

\begin{lemma} \label{lemma_cyclic_orbits_m=0}
Suppose $v_p(j)\geq n-1-v_p(r)$. Then 
\[ {\mathcal O}(j)=\{j\} \mbox{ and } \mbox{Stab}(j)=\mathbb{Z}/p^n\mathbb{Z}. \]
\end{lemma}
\begin{proof} Immediate from Lemma \ref{lemma_cyclic_orbits} since $m=0$. Note additionally that there are $p^{1+v_p(r)}$ such $j$ since 
\[\{j: v_p(j)\geq n-1-v_p(r) \}=\{j' p^{n-1-v_p(r)} : 0\leq j^{\prime} < p^{1+v_p(r)} \}. \]
\end{proof}

Having computed orbits and stabilizers, we are now able to give Wedderburn-Artin decompositions. 

\begin{theorem} \label{theorem_cyclic_p_power_wedderburn}
There is an isomorphism of $ K $-algebras
\[H_{r} \cong K^{p^{1+v_p(r)}} \times \prod_{m=1}^{ n-1-v_p(r)}\left( K_{m} \right)^{p^{v_p(r)}(p-1)}.\]
\end{theorem}
\begin{proof} Suppose first that $r=0$. Then $H_{r}=K[N]$, and since $\zeta\in K$ it follows that 
\[H_{p^{n-1}}\cong K^{p^n} = K^{p^{1+v_p(p^{n-1})}}.\]
Now let $0 < r < p^{n-1}$. Pick $0 < m \leq n-1-v_p(r)$. From Lemma \ref{lemma_cyclic_orbits_m>0}, we know that there are $p^{v_p(r)}(p-1)$ distinct orbits whose stabilizer is $ \langle \sigma^{p^m} \rangle $. Since $L^{\langle \sigma^m \rangle }=K_{m}$, the Wedderburn-Artin decomposition contains $p^{v_p(r)}(p-1)$ copies of $K_ {m}$. Therefore,
\[H_{r}\cong K^{p^{1+v_p(r)}} \times \prod_{m=1}^{ n-1-v_p(r)}\left( K_{m} \right)^{p^{v_p(r)}(p-1)} \mbox{ as $ K $-algebras.}\]
\end{proof}

\bibliography{../structure}
\bibliographystyle{plain}
\end{document}